\documentclass[11pt,reqno]{amsart}
\usepackage{graphicx, cite} 
\usepackage{amsfonts,amssymb,amsmath,mathtools}
\usepackage{amsthm}

\newcommand\blfootnote[1]{%
  \begingroup
  \renewcommand\thefootnote{}\footnote{#1}%
  \addtocounter{footnote}{-1}%
  \endgroup
}

\def\dj{d\kern-0.4em\char"16\kern-0.1em}
\def\Dj{\mbox{\raise0.3ex\hbox{-}\kern-0.4em D}}

\def\be{\begin{equation}}
\def\ee{\end{equation}}

\def\t{\tau}
\def\s{\sigma}

\def\suml{\sum\limits}

\def\dss{\displaystyle}

\newcommand{\supp}{\operatorname{supp}}

 \def\D{\mathcal{D}}
 \def\E{\mathcal{E}}
 \def\Rd{\mathbb{R}^d}
 
\def\N{\mathbb{N}}

\newtheorem{te}{Theorem}
\numberwithin{te}{section}
\newtheorem{prop}[te]{Proposition}

\newtheorem{lema}[te]{Lemma}

\newtheorem{de}[te]{Definition}

\newtheorem{rem}[te]{Remark}

\begin{document}


\title{Band-limited wavelets beyond Gevrey regularity}
\author[Nenad Teofanov]{Nenad Teofanov$^*$}\blfootnote{$^*$ corresponding author}

\address{\hspace{-\parindent} Nenad Teofanov, Department of Mathematics and Informatics, Faculty of Sciences
	University of Novi Sad, Serbia}

\email{nenad.teofanov@dmi.uns.ac.rs}

\author[Filip Tomi\'c]{Filip Tomi\'c}

\address{\hspace{-\parindent} Filip Tomi\'c, Faculty of Technical Sciences,
University of Novi Sad, Serbia}

\email{filip.tomic@uns.ac.rs}

\author[Stefan Tuti\'c]{Stefan Tuti\'c}

\address{\hspace{-\parindent} Stefan Tuti\'c, Department of Mathematics and Informatics, Faculty of Sciences
	University of Novi Sad, Serbia}

\email{stefan.tutic@dmi.uns.ac.rs}



\maketitle

\begin{abstract}
It is known that a smooth function of exponential decay  at infinity can not be an orthonormal wavelet.
Dziuba\'nski and Hern\'andez constructed smooth orthonormal wavelets of Gevrey type subexponential decay.
We weaken  the Gevrey type decay and construct orthonormal wavelets of  subexponential decay
related to the so-called extended Gevrey classes. The virtue of our construction is that prescise asymptotics of  functions from such classes can be given in terms of the Lambert $W$ function.

\end{abstract}

%
%

\section{Introduction}

We complment the studies of orthonormal wavelets of subexponential decay given in \cite{Hernandez} and \cite{Moritoh}, by extending their classes of to smooth functions with subexponential decay given in terms of  the Lambert $W$ function (cf.  \cite{LambF, Mezo}).
Recall, a function $ \psi \in L^2 (\mathbb{R}) $
is called an {\em orthonormal wavelet} (ONW)
if the set
$$ \{\psi_{m, n} : m \in \mathbb{Z}, n \in \mathbb{Z}\}
$$
is an orthonormal basis of $ L^2 (\mathbb{R}),$ where
$$
\psi_{m, n}(x) = 2^{\frac{m}{2}} \psi(2^m x - n), \qquad m,n \in \mathbb{Z},\  x \in \mathbb{R}.
$$
Regularity and decay properties of the function $ \psi $ play central role in applications. The construction of an
orthonormal wavelet $\psi \in \mathcal{S} (\mathbb{R})$
(see Section \ref{Sec2} for the notation)
goes back to Y. Meyer, \cite{Mey, LM}.
However, if  $\psi \in C^\infty (\mathbb{R})$ with all derivatives bounded, and if
$\psi$ is of exponential decay at infinity, i.e.
\be
\label{Eq:exp-decay}
| \psi (x) | \leq C e^{-h |x|}, \qquad x \in \mathbb{R},
\ee
for some $h,C>0$, then $\psi$ can not be  an orthonormal wavelet \cite[Corollary 5.5.3]{Dau}.

\par

Our study complements the Dziuba\'nski and Hern\'andez construction of
orthonormal wavelets $ \psi $ with subexponential decay,
\be
\label{Eq:subexp-decay}
| \psi (x) | \leq  C e^{- |x|^{\frac{1}{\sigma}}}, \qquad x \in \mathbb{R}, \; \sigma > 1,
\ee
for some $ C= C(\sigma)>0$, \cite{Hernandez}.
The functions considered in  \cite{Hernandez} are related to the Gevrey type regularity.

By using Komatsu's approach to ultradistributions, Moritoh and Tomoeda  proved in \cite{Moritoh} that a similar construction to the one given in
\cite{Hernandez} may lead to the following  decay of $\psi $:
\be
\label{Eq:moritoh-decay}
| \psi (x) | \leq  C e^{-
|x|/l_{n,\sigma} (|x|)
}, \qquad |x| \;\; \text{large enough},  \; \sigma>1,
\ee
for any $n = 1,2,\dots$, where
$$
l_{n,\sigma} (|x|) = (\log_1 |x|)(\log_2 |x|) \dots (\log_{n-1} |x|)(\log_n |x|)^\sigma,
$$
$\log_1 |x|=\log |x|$, $|x| >0$.

\par

The decay in  \eqref{Eq:moritoh-decay} is closer to exponential decay \eqref{Eq:exp-decay} than
the subexponential decay given by \eqref{Eq:subexp-decay}. We extend the class of  ONW considered in \cite{Hernandez} and \cite{Moritoh}, and obtain an estimate of the form
\begin{multline}
\label{Eq:Lambert-decay}
| \psi (x) | \leq C
 e^{ -h \left ( \ln^{2}(|x|) / W  (\ln |x|) \right ) } \\
\asymp
 e^{ -h \left ( \ln^{2}(|x|) / \ln  (\ln |x|) \right ) }
, \qquad  |x| \;\; \text{large enough},
\end{multline}
for every $h>0$ and some $ C= C(h)>0$,
(cf. Theorem  \ref{Th:first-estimate}  for a more general result).
Here  $W$ denotes the Lambert function. We note that
\eqref{Eq:Lambert-decay} is less restrictive than \eqref{Eq:subexp-decay}.

\par

The  Lambert $W$ function is defined as the inverse of $z e^{z}$, $z\in \mathbb{C}$,
and we consider here $ W(x)$ the restriction of its principal branch to $ [0,\infty)$.
It is a convenient tool for the description of asymptotic behavior in different contexts.
We refer to \cite{LambF} for a review of some applications of
the Lambert $W$ function in pure and applied mathematics, and to recent monograph
\cite{Mezo} for more details and generalizations.
Let us briefly and incompletely mention its use in different research fields, such as
epidemiology \cite{BDG}, cosmology \cite{Yang}, ecology \cite{PSSMA}, and numerical analysis \cite{Roos}.

\par

Our construction is based on the  properties of  extended Gevrey classes
$ {\E}_{\t,\s}({\mathbb R})$, $\t>0$, $\s>1$, see Subsection \ref{subsec:extendedGevrey} for the definition.
A surprising feature of $ {\E}_{\t,\s}({\mathbb R})$ is that the precise asymptotics of the corresponding associated function
$ T_{\tau, \sigma}$ (see \eqref{asocirana}), is related to the Lambert $W$ function, \cite{PTT-03}.
This is recently used to prove the surjectivity of the Borel map in sectors of the complex plane for
the certain ultraholomorphic classes \cite{Javier02}, and in the study  of ultradifferentiable functions ''beyond geometric growth factors'' \cite{Javier01}. The classes
$ {\E}_{\t,\s}({\mathbb R})$ are closely related to the general construction of ultradifferentiable functions defined via weight matrices, \cite{RS, Rainer, TT0}. We also note that  $ {\E}_{1,2}({\mathbb R})$ provides well-posedness
for certain systems of strictly hyperbolic PDE's, see  \cite{CL}.

\par

We consider wavelets  $\psi $ whose Fourier transforms $\widehat \psi $
are compactly supported functions from ${\E}_{\t,\s}({\mathbb R})$.
Recall, if $ \hat f $ has compact support
then $f $ is called a band-limited function, and if $\psi $ is a band-limited ONW
such that  $ \hat \psi \in C^\infty (\mathbb{R})$, then by
the Paley-Wiener theorem it follows that $\psi$ is analytic function.
To study non band-limited ONW $\psi \in C^\infty (\mathbb{R})$ which are less regular than any Gevrey function
a different approach should be used, see e.g. \cite{Fukuda}.

%
%

\par

In the same way as it is done in  \cite{Hernandez, Moritoh}, an example of the orthonormal wavelet $\psi$ is given by the explicit construction of  $\widehat \psi $. Moreover, by the method used in the construction $\widehat \psi $ does not belong to any Gevrey class, see Remark \ref{rem-sharp}. Consequently, $\psi$ does not satisfy \eqref{Eq:subexp-decay} (and therefore \eqref{Eq:moritoh-decay}), which gives an extension of ONW of  subexponential decay observed in the literature so far.

\par

The paper is organized as follows.
In Section \ref{Sec2} we collect necessary background information to set the stage for main results that are
collected in Section \ref{sec3}. To prove the main Theorem \ref{Th:first-estimate},
we use the construction of wavelets (Lemma \ref{Prop:F-of-psi}) and the Paley-Wiener type theorem
(Proposition \ref{NASPW}).
In addition to the decay of orthonormal wavelet $\psi$ given in Theorem  \ref{Th:first-estimate} \emph{i)},
we present a more precise estimate in  Theorem \ref{Th:first-estimate} \emph{ii)},
which clearly reveals ultra-analyticity of $\psi$. This type of ultra-analytic estimates is not considered in
\cite{Hernandez} and \cite{Moritoh}.


\section{Preliminaries} \label{Sec2}

The main goal of this section is to introduce functions with extended Gevrey regularity. As a preparation we first
set the notation and then introduce the Lambert function, defining sequences and
associated functions.

\subsection{Notation}
We use the standard notation: $\mathbb{N}^d$, $\mathbb{R}^d$, $\mathbb{C}^d$, denote the sets of $d-$dimensional positive integers, real numbers, and complex numbers, respectively. When $d=1$ we omit the upper index, and $\mathbb{N}_0 = \mathbb{N} \cup \{0\}$. The length of a multi-index $\alpha \in \mathbb{N}_0 ^d$ is denoted by $ |\alpha| = \alpha_1 + \alpha_2 + \dots + \alpha_d$.
We write $ L^p (\mathbb{R}^d) $, $ 1\leq p\leq \infty$,  for the Lebesgue spaces, and  $ \mathcal{S} (\mathbb{R}^d)$ denotes the Schwartz space of infinitely smooth ($ C^\infty (\mathbb{R}^d)$) functions which, together with their derivatives, decay at infinity faster than any inverse polynomial. $f = O (g) $ means that $ |f(x)| \leq C |g(x)| $ for some $C>0$ and $x$ in the intersection of domains for $f$ and $g$.
If $f=O(g)$ and $g=O(f)$, then we write $f\asymp g$.
The Fourier transform of $f \in L^1 (\mathbb{R}^d)$ given by
$$
\widehat{f} (\xi) := \int_{\mathbb{R}^d} f(x) e^{ i x \cdot \xi} dx, \qquad \xi \in \mathbb{R}^d,
$$ 
extends to $ L^2 (\mathbb{R}^d)$ by standard approximation procedure.

\subsection{The Lambert function} \label{subsec:lambert}
The \emph{Lambert $W$ function} is defined as the inverse of $z e^{z}$, $z\in {\mathbb C}$. By $W(x)$, $x\geq 0$, we denote the restriction of its principal branch, and we review some of its basic properties as follows:
\begin{itemize}
\item[$(W1) \quad$] $W(0)=0$, $W(e)=1$, $W(x)$ is continuous, increasing and concave on $[0,\infty)$,
\vspace{0.2cm}
\item[$(W2) \quad$] $W(x e^{x})=x, $ and $ \quad x=W(x)e^{W(x)}$,  $x\geq 0$,
\vspace{0.2cm}
\item[$(W3) \quad $] $W$ can be represented in the form of the absolutely convergent series
$$
W(x)=\ln x-\ln (\ln x)+\sum_{k=0}^{\infty}\sum_{m=1}^{\infty}c_{km}\frac{(\ln(\ln x))^m}{(\ln x)^{k+m}},\quad x\geq x_0>e,
$$
with suitable constants $c_{km}$ and  $x_0 $, wherefrom  the following  estimates hold:
\be
\label{sharpestimateLambert}
\ln x -\ln(\ln x)\leq W(x)\leq \ln x-\frac{1}{2}\ln (\ln x), \quad x\geq e.
\ee
\end{itemize}

Note that $(W2)$ implies
$$
W(x\ln x)=\ln x,\quad x>1.
$$
By using $(W3)$ we obtain
$$
W(x)\sim \ln x, \quad x\to \infty,
$$
and therefore
$$
W(C x)\sim W(x),\quad x\to \infty,
$$
for any $C>0$.
We refer to \cite{HoHa, LambF, Mezo} for more details about the Lambert $W$ function.

\subsection{Defining sequences} \label{sequences}
The well known Komatsu's approach to ultradistributions, see \cite{Komatsuultra1},is based on defining sequences
$M_p$, $ p \in \N$, which satisfy some of the conditions:

\begin{itemize}
	\item[$(M.1) \quad$] $\dss (M_p)^2\leq M_{p-1} M_{p+1} $, $ \quad p\in \N$,\\
	\item[$(M.2) \quad$]
	$ (\exists A,C>0) \quad  M_{p}\leq A C^{p+q} \;  M_p M_{q}, \quad p,q\in \N_0,$ \\
	\item[$(M.2)' \quad$]
	$ (\exists A,C>0) \quad M_{p+1} \leq A C^{p}M_p,  \quad p\in \N_0.$
	\item[$(M.3)'\quad$]
	$\displaystyle \suml_{p=1}^{\infty}\frac{M_{p-1}}{M_p} <\infty.$
\end{itemize}

When $ \t > 1 $, the Gevrey sequence $M_p = p!^{\t}$, $ p \in \N$, is an example of of a sequence
which satisfy all of the above mentioned conditions.

We extend Gevrey sequences by introducing an additional parameter $\s $, and consider
sequences of the form $M^{\t,\s}_p=p^{\t p^\s}$,
$ p \in \N$, $M^{\t,\s}_0 = 1$, $\t>0$, $\s \geq 1$.

It can be proved (see e.g. \cite{PTT-01}) that such sequences satisfy
the following properties:

\begin{itemize}
\item[$(M.1) \quad$] $\dss ({M_p^{\t,\s}})^2\leq {M_{p-1}^{\t,\s}}{M_{p+1}^{\t,\s}}$, $ \quad p\in \N$,\\

\item[$\widetilde{(M.2)'}\quad$]  $(\exists C>0)\quad M_{p+1}^{\t,\s}\leq C^{p^{\s}}M_p^{\t,\s}$,
$\quad p\in \N_0$,\\

\item[$\widetilde{(M.2)}\quad$]  $ (\exists C>0)$\quad $M_{p+q}^{\t,\s}\leq C^{p^{\s}+q^{\s}}M_p^{2^{\s-1}\t,\s}M_q^{2^{\s-1}\t,\s},\quad p,q\in \N_0,$ \\

\item[$(M.3)'\quad$] If $\s > 1$, or if $\s = 1$ and $\t>1 $, then
$\displaystyle \suml_{p=1}^{\infty}\frac{M_{p-1}^{\t,\s}}{M_p^{\t,\s}} <\infty.$
\end{itemize}

In fact, $(M.3)'$ follows from the estimate
\be
\label{NizOcena}
\frac{M_{p-1}^{\t,\s}}{M_p^{\t,\s}}\leq \frac{1}{(2p)^{\tau (p-1)^{\s-1}}},\qquad p\in \N.
\ee

Notice that $(M.2) \Rightarrow (M.2)'$.
When $ \s = 1$ i.e.  $M^{\t,1}_p=p^{\t p} \asymp p! ^\t$, $\t>0$, we obtain the classical Gevrey sequences. Then $\widetilde{(M.2)}$ and $\widetilde{(M.2)'}$ reduce to  $(M.2)$ and $(M.2)'$, respectively, and  $(M.3)'$ holds if and only if $\t>1$.

\par

\begin{rem}
The estimate $\widetilde{(M.2)}$ is sharp in the sense that ${M_p^{\t,\s}}$ does not satisfy $(M.2)$ for any choice of $\t>0$ and $\s>1$. More precisely, $\widetilde{(M.2)'}$ holds for $C>1$ and therefore ${M_p^{\t,\s}}$ fails to satisfy a weaker condition $(M.2)'$ (see \cite{PTT-01}). This implies that classes of smooth functions generated by ${M_p^{\t,\s}}$, see Subsection \ref{subsec:extendedGevrey},  are not ultradifferentiable in the sense of Komatsu, \cite{Komatsuultra1}. However, it is recently proved that such classes fit well to the theory of \emph{weight matrix spaces} (see \cite{TT0, Rainer, RS}).
\end{rem}

\subsection{Associated functions} \label{subsec:weights}

The associated function $T_{\t,\s}$ to the sequence $ M_p^{\t,\s} = p^{\t p^\s}$, $M^{\t,\s}_0 = 1$, $\t>0$, $\s>1$, is defined by
\be \label{asocirana}
\dss T_{\t,\s}(k)=\sup_{p\in \N_0}\ln_+\frac{k^{p}}{M_p^{\t,\s}},\quad \;\; k>0,
\ee
where $\ln_+ x=\max\{0,\ln x\}$, $x>0$.

It turns out that the precise asymptotic behavior of $T_{\t,\s}$ an infinity can be expressed via the Lambert W function as follows.
\begin{prop}
\label{TeoremaAsocirana}
Let $\t>0$, $\s>1$,  $M^{\t,\s}_p=p^{\t p^\s}$, $ p = 1,2, \dots$, let $ T_{\t,\s}$ be the
associated function to the sequence ${M_p^{\t,\s}}$, and put
\be
 T_{\s}(k) :=  \frac{\ln ^{\frac{\s}{\s-1} }k }{W ^{\frac{1}{\s-1}}(\ln  k)}, \quad
 k \; \text{large enough}.
\ee
Then the following estimates hold:
\be
\label{KonacnaocenaAsocirana}
B_\s\,\t^{-\frac{1}{\s-1}}T_{\s}(k) +\widetilde{B}_{\t,\s}
\leq T_{\t,\s}(k)
\leq A_{\s}\,\t^{-\frac{1}{\s-1}} T_{\s}(k) + \widetilde{A}_{\t,\s},\,\,
\ee
for large enough $k>0$, and suitable constants $A_{\s}, B_{\s}, \widetilde{A}_{\t,\s},\widetilde{B}_{\t,\s}>0$.
\end{prop}

The proof follows from \cite[Proposition 2]{TT0}, and estimates (30) given there. We note that \eqref{KonacnaocenaAsocirana} implies
\be
\label{AsociranaSigma}
 T_{\t,\s}(k)\asymp \tau^{-\frac{1}{\s-1}} \frac{\ln ^{\frac{\s}{\s-1} }k }{W ^{\frac{1}{\s-1}}(\ln  k)}
 =\tau^{-\frac{1}{\s-1}} T_{\s}(k), \quad
 k \; \text{large enough},
\ee
where the hidden constants depend on $\s$ only.

\par

We note that since $T_{\t,\s}$ is an increasing function, \eqref{AsociranaSigma} implies that there exists $A>0$ such that
\be
\label{AlmostIncreasing}
\displaystyle T_{\s}(k)\leq A\, T_{\s}(k+a),\quad a>0.
\ee

\subsection{Extended Gevrey classes} \label{subsec:extendedGevrey}
Although in Section \ref{sec3} we are interested in one-dimensional wavelets, here we decided to consider functions over $\mathbb{R}^d$.

Recall that the Gevrey space ${\mathcal G}_t ( \mathbb{R}^d)$, $ t >1$,
consists of functions $  \phi \in C^\infty ( \mathbb{R}^d)$ such that
such that for every compact set   $K\subset \subset \Rd$ there are constants $C_K > 0 $ and $h>0$ satisfying
\be \label{Eq:Gevrey-class}
|\partial^{\alpha} \phi (x)| \leq  C_K h^{|\alpha|} |\alpha|!^t.
\ee
for all $x \in K $ and for any $\alpha \in \N_0 ^d$.

In a ismilar fashion we introduce classes related to defining sequences $M^{\t,\s}_p=p^{\t p^\s}$, $\tau>0$, $\s>1$.

\begin{de} \label{def:extendedGervey} Let there  be given $\tau>0$, $\s>1$, and let $M^{\t,\s}_p=p^{\t p^\s}$, $ p = \mathbb{N} $, $M^{\t,\s}_0 = 1$.

The extended Gevrey class of Roumieu type ${\E}_{\{\t, \s\}}( \mathbb{R}^d)$ is the set of all
$\phi \in  C^{\infty}( \mathbb{R}^d)$ such that for every compact set   $K\subset \subset \Rd$ there are constants $C_K > 0 $ and $h>0$ satisfying
\begin{equation} \label{NewClassesInd}
|\partial^{\alpha} \phi (x)| \leq C_K h^{|\alpha|^{\s}}  M_{|\alpha|} ^{\t,\s},
\end{equation}
for all $x \in K $ and for all $ \alpha \in  \mathbb{N}_0 ^d.$

The extended Gevrey class of Beurling type ${\E}_{(\t, \s)}( \mathbb{R}^d)$ is the set of all
$\phi \in  C^{\infty}( \mathbb{R}^d)$ such that
for every compact set   $K\subset \subset \Rd$ and for all $h>0$ there is a constant $C_{K,h} > 0 $  satisfying
\begin{equation} \label{NewClassesProj}
|\partial^{\alpha} \phi (x)| \leq C_{K,h} h^{|\alpha|^{\s}}  M_{|\alpha|} ^{\t,\s},
\end{equation}
for all $x \in K $ and for all $ \alpha \in  \mathbb{N}_0 ^d.$
\end{de}

We refer to \cite{PTT-01} for the topologies in ${\E}_{\{\t, \s\}}( \mathbb{R}^d)$ and ${\E}_{(\t, \s)}( \mathbb{R}^d)$.
Note that   \eqref{NewClassesInd}  and \eqref{NewClassesProj} imply
${\E}_{(\t, \s)}( \mathbb{R}^d) \hookrightarrow {\E}_{\{\t, \s\}}( \mathbb{R}^d)$,
where $\hookrightarrow$ denotes the continuous and dense inclusion.

We use abbreviated notation $ \t,\s $ for $\{\t,\s\}$ or $(\t,\s)$.

The set of functions $ \phi \in {\E}_{\t,\s}( \mathbb{R}^d)$
whose support is contained in some compact set is denoted by  ${\D}_{\t, \s} ( \mathbb{R}^d)$.

We also introduce
\be
\label{KlasePreseci}
{\E}_{\s}( \mathbb{R}^d)=\bigcap_{\t>0}{\E}_{\t,\s}( \mathbb{R}^d)
\quad \text{and }\quad  {\D}_{\s}( \mathbb{R}^d)=\bigcap_{\t>0}{\D}_{\t,\s}( \mathbb{R}^d),
\ee
endowed with projective limit topologies.

\par

If $ \t > 1 $ and $\s = 1$, then we have $  M_{|\alpha|} ^{\t,1} = |\alpha|^{\t |\alpha|} \asymp  |\alpha|!^t$, so that
$$
{\E}_{\{\t, 1\}}(\mathbb{R}^d)={\E}_{\{\t\}}(\mathbb{R}^d) = {\mathcal G}_{\t} (\mathbb{R}^d),
$$
and $\D_{\{\t,1\}}(\mathbb{R}^d)=\D_{\{\t\}}(\mathbb{R}^d)$ is its subspace of compactly supported functions in $\E_{\{\t\}}(\mathbb{R}^d)$.

\par

By the definition it follows that
\begin{equation*}
\label{Theta_S_embedd-2}
{\E}_{\tau_0, {\s_1}}(\mathbb{R}^d)\hookrightarrow \cap_{\tau> \tau_0} {\E}_{\tau, {\s_1}}(\mathbb{R}^d)
\hookrightarrow {\E}_{\tau_0, {\s_2}}(\mathbb{R}^d),
\end{equation*}
for any $\tau_0>0$ whenever $\s_2>\s_1\geq 1$. In particular,
\begin{equation*}
\label{GevreyNewclass}
\cup_{t>1}  {\mathcal G}_t (\mathbb{R}^d) \hookrightarrow {\E}_{\tau, \s}(\mathbb{R}^d)
\hookrightarrow  C^{\infty}(\mathbb{R}^d), \;\;\; \tau>0, \; \s>1,
\end{equation*}
so that  the regularity in ${\E}_{\tau, \s}(\mathbb{R}^d)$ can be thought of as an \emph{extended Gevrey regularity}.

\par

If $0<\t\leq 1$, then $ {\E}_{\t, 1}(\mathbb{R}^d)$ consists of quasianalytic functions. In particular, $\dss
\D_{\t,1}(\mathbb{R}^d)=\{0\}$ when $0<\t\leq 1$, and ${\E}_{\{1, 1\}}(\mathbb{R}^d)= {\E}_{\{1\}}(\mathbb{R}^d)$ is the set of analytic functions on $\mathbb{R}^d$.

\par

The non-quasianalyticity condition $(M.3)'$ provides the existence of partitions of unity in  $\E_{\t,\s}(\mathbb{R}^d)$ which can be formulated as follows.

\begin{lema} [\cite{PTT-01}] \label{teoremaKompaktannosac}
Let $\t>0$ and $\s>1$. Then there exists a compactly supported function $\phi\in {\E_{\t,\s}}(\mathbb{R}^d)$ such that $0\leq\phi\leq 1$ and $\int_{\Rd}\phi\,dx=1$.
\end{lema}

Of course, any compactly supported Gevrey function from $\D_{\{\t\}}(\mathbb{R}^d)$ belongs to $ {\D}_{\{\t, \s\}}(\mathbb{R}^d)$ as well.
However, the proof of Lemma \ref{teoremaKompaktannosac} from \cite{PTT-01} provides a compactly supported function in ${\D}_{\{\t, \s\}}(\mathbb{R}^d)$  which
does not belong to   ${\D}_{\{t\}}(\mathbb{R}^d)$, for any $t>1$. We will revisit the construction in Section \ref{sec3}.


We conclude the section with some properties which will be used in Section \ref{sec3}.

\begin{prop}
\label{OsobineKlasa}
Let there be given $\tau>0$, $\s>1$, $M^{\t,\s}_p=p^{\t p^\s}$, $ p = \mathbb{N}$,
$M^{\t,\s}_0= 1$,
and let ${\E}_{\t,\s}(\mathbb{R}^d) $ be given by \eqref{NewClassesInd} or \eqref{NewClassesProj}.
Then the following is true.
\begin{itemize}
\item[$i)$] ${\E}_{\t,\s}(\mathbb{R}^d) $ is closed under the {pointwise} multiplication.
\item[$ii)$] ${\E}_{\t,\s}(\mathbb{R}^d) $ is closed under finite order derivation.
\item[$iii)$]  $\E_{\t,\s}(\mathbb{R}^d)$ is closed under superposition. More precisely, if $F(x)$ is an entire function and $f(x)\in \E_{\t,\s}(\mathbb{R}^d)$ then $F(f(x))\in \E_{\t,\s}(\mathbb{R}^d)$.
\item[$iv)$] $\E_{\t,\s}(\mathbb{R}^d)$ is invariant under translations and dilations.
\end{itemize}
\end{prop}

We note that conditions $(M.1)$ and $\widetilde{(M.2)'}$ play an essential role in the proof of Proposition \ref{OsobineKlasa},  which can be found in e.g. \cite{PTT-01, PTT-03, PTT-031}.

More details on extended Gevrey regularity with references to applications can be found in \cite{PTT-01, PTT-03, PTT-031, TT0, TT1}.



\section{Main results} \label{sec3}

In \cite{Hernandez} Dziuba\'nski and Hern\'andez constructed a band-limited orthonormal wavelet $\psi \in C^\infty ({\mathbb R})$ which belongs to the intersection of all Gevrey classes in the frequency domain, i.e., $\dss \widehat{\psi}\in\bigcap_{t>1} {\mathcal G}_t ({\mathbb R}) $. This is done through the refinement of
related constructions from \cite{HW, Mandel, H}.
By the Paley--Wiener theorem for Gevrey classes (cf. \cite[Theorem 1.6.1]{Rodino})
it follows that $\psi $ has a sub-exponential decay in the space (configuration) domain, i.e.
$$
| \psi (x) | \leq C e^{-|x|^{1/t}}, \qquad x \in \mathbb{R}, \qquad t>1,
$$
for some $C>0$, which is close to the critical decay \eqref{Eq:exp-decay}.

It turned out that the construction from \cite{Hernandez} can be further refined by using ultradifferentiable functions given by sequences $M_p$, $p\in \N_0$, which satisfy $(M.1)$, $(M.2)'$ and $(M.3)'$, see Subsection \ref{sequences}. Namely, Moritoh and Tomoeda considered  sequences of the form
$$\dss L^{n,\s}_p=p!(\log_n p)^{\s}\prod_{k=1}^{n-1}\log_k p,\quad p\in \N_0,$$ where $n\in\N_0$ and $\s>1$ ($\log_1 p=\log p$), \cite{Moritoh}. It is easy to see that for every $\s,t>1$ and $n\in \N_0$ there exists $C>0$ such that
$$
L^{n,\s} _p \leq p!^t \frac{(\log_2 p)^{n+\s}}{p!^{t-1} }\leq C p!^t,\quad p\in \N_0.
$$
Therefore classes generated by $L^{n,\s} _p$ are contained in the intersection of all Gevrey classes. The decay of wavelets $ \psi $ from \cite{Moritoh} is given by
$$
| \psi (x) | \leq  C e^{- \frac{|x|}{ l_{n,\sigma} (|x|)}}, \qquad |x| \;\; \text{large enough},  \; \sigma>1, \; n \in \N,
$$
and $ l_{n,\sigma} (\cdot) =(\log_n (\cdot))^\sigma \prod_{k=1}^{n-1}\log_k  (\cdot) $.

\par

Next we extend the classes of orthonormal wavelets from \cite{Hernandez} and \cite{Moritoh},
and construct a band-limited orthonormal wavelet $\psi\in C^\infty ({\mathbb R})$  with decay given in terms of the Lambert $W$ function.

We first construct a cutoff function  $\phi \in \mathcal{D}_{ \sigma} ({\mathbb R})$ as follows.

\begin{lema} \label{NasaCutoff}
Let $\s>1$. For every $a>0$ there exists
$\phi_a \in \mathcal{D}_{ \sigma}({\mathbb R})$ $=\cap_{\t>0}{\D}_{\t,\s}({\mathbb R}) $ such that $\phi_a  \geq 0$, $\supp\phi_a \subset [-a,a]$, and $\int_{\mathbb R } \phi_a (x) \, dx = 1$.
\end{lema}

\begin{proof}
Since $ \mathcal{D}_{ \sigma}({\mathbb R})$ is closed under dilations and multiplications by a constant, it is enough to show  the result for $a=1$.

The essential part of the construction is to carefully choose a sequence $a_p $,  $p \in \mathbb N _0$, which will give rise to decay properties specified by the definition of extended Gevrey classes.

Consider first the (double indexed) sequence
$$
a(m,p) = \frac{1}{(2 (p + 1))^{\frac{1}{m}p^{\sigma - 1}}} , \qquad m,p\in\mathbb N _0,
$$
for a given $\sigma >1$.
Since
$$
\sum_{p=1}^{\infty} \frac{1}{(2 (p + 1))^{\frac{1}{m}p^{\sigma - 1}}}< \infty
$$
for any  $m\in\mathbb N _0$, it follows that
there exists a sequence $N_m \in\mathbb N _0$ such that
$$
\sum_{n=N_m}^{\infty} \frac{1}{(2 (p + 1))^{\frac{1}{m}p^{\sigma - 1}}} < \frac{1}{2^m}.
$$
Thus the sequence  $a_p $, $p \in \mathbb N _0$, given by
$$
a_p := \frac{1}{(2 (p + 1))^{\frac{1}{m}p^{\sigma - 1}}}, \qquad N_m  \leq p < N_{m+1},
$$
satisfies
$$
\sum_{n=N_1}^{\infty} a_p \leq 1.
$$

The next step in the construction is to observe the sequence of functions
$$
\phi_{p} := f_{a_{N_1}}* f_{a_{N_1 + 1}}*\dots * f_{a_{p}},
$$
where $f\in C^{\infty} (\mathbb R )$ is a non-negative and even function such that
$\supp f \in [-1,1]$, $\int_{-1}^1 f(x) \, dx = 1$,
and $f_{a} (x) = \frac{1}{a} f\left(\frac{x}{a}\right)$. From $\int_{\mathbb R} f_a (x) \, dx = 1$ (and $\int |(f*g)(x)| \, dx = \int |f (x)| \, dx \int |g (x)| \, dx$) it follows that
$$\supp\phi_{p} \subset [-1,1] \quad \text{ and } \quad
\int_{-1}^1 \phi_{p}(x) \, dx = 1, \qquad
p\in\mathbb N _0.
$$

Let us show that for every given $\tau > 0$ there exists $C_{\tau} > 0$ such that
$$ |\phi_p ^{(n)} (x)| \leq C_{\tau} ^{n^{\sigma}} n^{\tau n^{\sigma}}, \qquad x\in\mathbb R,$$
for every $n\in\mathbb N _0$ and every $p\geq p_0 (n, \tau)$.
Note that
$$ \phi_{p} ^{(n)} = f_{a_{N_1}}*\dots * f_{a_{N_m}}* f'_{a_{N_m + 1}}*\dots * f '_{a_{N_m + n}}*
f_{a_{N_m + n + 1}}*\dots * f_{a_{p}},$$
and
$$ \| f' _{a_p} \|_1 = \frac{1}{a_p}
\int_{\mathbb{R}} \frac{1}{a_p} \left| f' \left( \frac{x}{a_p} \right) \right | \, dx
\leq \frac{c}{a_p} \leq c \, (2 (p + 1))^{\frac{1}{m}p^{\sigma - 1}}$$
when $p\geq N_m$.

Let $n\in\mathbb N _0$ be given, and consider the sequence $\tau_m = \frac{1}{m}$, $m\in\mathbb N _0$. Then for any given $\tau > 0$ we can choose $m\in\mathbb N _0$ so that $\tau_m < \tau$, and $p\in\mathbb N _0$ so that $N_m + n < p$.

By using the fact that ${\widetilde{(M.2)'}}$ implies
$$
M^{\frac{1}{m}, \s} _{p+q} \leq \tilde{C}^{p^\s} M^{\frac{1}{m}, \s} _{p}
$$
for some $\tilde{C}= \tilde{C}(q)>0$, we obtain

\begin{align*}
|\phi_p ^{(n)} (x)| &\leq c^n \, 2^{\frac{1}{m} \sum_{k=1}^n (N_m+k)^{\sigma - 1}} \prod_{k=1}^n (N_m + k + 1)^{\frac{1}{m} (N_m+k)^{\sigma - 1}}\\
&\leq c^n 2^{\frac{1}{m} n(N_m + n)^{\sigma - 1}} (N_m + n + 1)^{\frac{1}{m} (N_m + n)^{\sigma - 1} n}\\
&\leq c^n 2^{\frac{1}{m} (N_m + n)^{\sigma }} (N_m + n + 1)^{\frac{1}{m} (N_m + n + 1)^{\sigma }}\\
&{\leq} c^n 2^{(N_m + n)^{\sigma}} \tilde{C} ^{n^{\sigma}} \, n^{\frac{1}{m} n^{\sigma}}
\leq C^{n^{\sigma}} n^{\tau n^{\sigma}},
\end{align*}
where  $C$ depends on $\t$.

The sequence $ \{ \phi_p ^{(n)} \; | \; p = N_1, N_2, \dots \} $ is  a Cauchy sequence for every  $ n = 0,1,2,\dots$. Thus it converges to a function $ \phi $ which satisfies
$$
|\phi ^{(n)} (x)| \leq C^{n^{\sigma}} n^{\tau n^{\sigma}},
$$
for every $\t > 0$. Therefore $\phi \in \mathcal{D}_{ \sigma}({\mathbb R})$, and by the construction $\phi  \geq 0$, $\supp\phi \subset [-1,1]$ and $\int_{\mathbb R } \phi (x) \, dx = 1$, which completes the proof.
\end{proof}

\begin{rem} \label{rem-sharp}
We note that the construction in Lemma \ref{NasaCutoff} is sharp in the sense that $\phi$ does not belong to any Gevrey class, i.e. $\phi \not\in \bigcup_{t>1}  {\mathcal G}_t ({\mathbb R})$. In fact, the regularity of $\phi $ is related to the minimal size of its support. We refer to the proof of  \cite[Lemma 1.3.6.]{H} for details.
\end{rem}

Next we modify the construction of ONW from  \cite{Hernandez} to obtain ONW in the framework of the extended Gevrey regularity. In the proof of Lemma \ref{Prop:F-of-psi} we use the properties of extended Gevrey classes given in  Proposition \ref{OsobineKlasa}.

\begin{lema}
\label{Prop:F-of-psi}
Let $\s>1$, and let $ \mathcal{D}_{\s} ({\mathbb R})$ be given by \eqref{KlasePreseci}. Then
there exists an orthonormal wavelet $ \psi \in C^\infty  ({\mathbb R})$
such that $\widehat{\psi}  \in \mathcal{D}_{\s} ({\mathbb R})$.
\end{lema}

\begin{proof}
Take arbitrary $\t>0$ and fix $\s>1$. Let $a \in (0, \pi/3)$ be fixed, and choose a function
$\phi_a \in \mathcal{D}_{\sigma} ({\mathbb R})$
as in Lemma \ref{NasaCutoff}, so that
$ \int_{\mathbb{R}} \phi_a (x) dx = \pi/2$.

Set
$$\theta_{a} (x) = \int_{-\infty}^x \phi_{a} (t) \, dt,\quad  x \in {\mathbb R},$$
so that $\theta_{a}$ is an increasing function, $\theta_a (x) = 0$ for $x<-a$, $\theta_a (x) = \frac{\pi}{2}$ for $x>a$, and $\theta_a \in \mathcal{E}_{ \sigma} ({\mathbb R})$, because $\theta_a ^{(n)} (x) = \phi ^{(n-1)} (x)$ for $n>1$, and
$$|\theta_a (x)| \leq (x+a) \sup \phi_a (x) \leq 2a \sup \phi_a (x),$$
and $\phi_a \in \mathcal{D}_{ \sigma} ({\mathbb R})$.

Then, by Proposition \ref{OsobineKlasa} \emph{iii)} the functions
$$
S_a (x) : = \sin (\theta_a (x)) \qquad \text{ and } \qquad C_a (x) : = \cos (\theta_a (x)), \quad  x \in {\mathbb R},
$$
belong to $\mathcal{E}_{ \sigma} ({\mathbb R})$.
$S_a (x)$ is an increasing function, and
$S_a (x-\pi) = 0$ for $x\leq \pi - a$, $S_a (x-\pi) = 1$ for $x > \pi + a$.
Similarly, $C_{2a} (x)$ is a decreasing function, and $C_{2a} (x-2\pi) = 1$ for $x\leq 2(\pi - a)$, $C_{2a} (x-2\pi) = 0$ for $x>2(\pi + a)$.

Next we define
\begin{equation}
\label{ba}
b_a (x) = S_a (x-\pi) C_{2a} (x-2\pi),  \quad  x \in {\mathbb R}.
\end{equation}
The function $b_a (x)$ is a bell shaped function, and $b_a (x) = 0$ when $x\leq \pi - a$ or $x\geq 2(\pi+a)$.
(In fact, the condition $0<a< \frac{\pi}{3}$ comes from $\pi + a < 2(\pi - a)$.)

\par

By Proposition \ref{OsobineKlasa} \emph{i)} and  \emph{iv)} it follows that
$b_a \in \mathcal{D}_{\sigma} ({\mathbb R})$
and, by a slight abuse of notation, we denote its even extension to $(-\infty , 0]$ by $b_a$  as well. Then it is proved in \cite{AWW} (see also \cite[Corollary 4.7, Chapter 1]{HW})
that the function $\psi$ defined by
\begin{equation}
\label{psi}
\widehat{\psi} (\xi) = e^{i \frac{\xi}{2}} b_a (\xi), \qquad \xi \in {\mathbb R},
\end{equation}
is an orthonormal wavelet in $L^2 (\mathbb{R})$, and $\psi \in C^\infty  ({\mathbb R})$.

Finally, that $\widehat{\psi}  \in \mathcal{D}_{\s} ({\mathbb R})$ follows from
\begin{align*}
|(e^{i\frac{\xi}{2}} b_a (\xi))^{(p)}|
&= \left| \sum_{k=0}^p \binom{p}{k} \frac{i^k}{2^k} e^{i \frac{\xi}{2}} b_a  ^{(p-k)} (\xi) \right|
\leq \sum_{k=0}^p \binom{p}{k} \frac{1}{2^k} |b_a ^{(p-k)} (\xi)|\\
&\leq \sum_{k=0}^p \binom{p}{k} \frac{1}{2^k} C^{(p-k)^{\sigma}} (p-k)^{\tau (p-k)^{\sigma}}\\
&\leq C^{p^{\sigma}} p^{\tau p^{\sigma}} \sum_{k=0}^p \binom{p}{k} \frac{1}{2^k} \leq C^{p^{\sigma}} p^{\tau p^{\sigma}} \sum_{k=0}^p \binom{p}{k} \\
&\leq \tilde{C} ^{p^{\sigma}}  p^{\tau p^{\sigma}},
\end{align*}
for some $C<0$ and $\tilde{C}  = 2C$.
\end{proof}

To prove Theorem \ref{Th:first-estimate} we  use the  Paley-Wiener theorem for extended Gevrey regularity.
A more general statement than Proposition \ref{NASPW} is given in \cite[Theorem 3.1]{PTT-03}.

\begin{prop}
\label{NASPW}
Let $\sigma > 1$. The Fourier transform of a function  $\dss f\in \, \mathcal{D_{ \sigma}}({\mathbb R})$ is an analytic function, and for every $h>0$ there exists a constant $C_{h }>0$ such that
\begin{align}
\label{ocenaTeorema1}
|\widehat{f}(\eta)| \leq C_{h } \, \displaystyle e^{ - h \left (
 \ln^{\frac{\s}{\s-1}}(|\eta|) / W^{\frac{1}{\s-1}}  (\ln(|\eta|))
\right ) }, \qquad |\eta| \quad \text{large enough},
\end{align}
where $W $ denotes the Lambert function.
\end{prop}

Notice that for $\s = 2 $ \eqref{ocenaTeorema1} amounts to \eqref{Eq:Lambert-decay} (with $\psi =\widehat{f}$).

\begin{proof} The analyticity of $f$ follows from the classical Paley-Wiener theorem, cf. \cite{H}. It remains prove \eqref{ocenaTeorema1}.

Take arbitrary $h>0$ and let $\dss f\in \, \mathcal{D_{ \sigma}}({\mathbb R})$ for $\s>1$, i.e.
$\dss f\in \, \mathcal{D_{\t, \sigma}}({\mathbb R})$ for every $\t>0$.
Let $K$ denote the support of $f$. Since $f\in {\D_{\frac{\t}{2}, \sigma}}({\mathbb R})$, by
Definition \ref{def:extendedGervey} we get the following estimate:
$$
|\eta^n\widehat f(\eta)|=  |\widehat{f^{(n)}}(\eta)|\leq C   \sup_{x\in K}|{f^{(n)}}(x)|\leq C^{n^\s +1}_1 n^{\frac{\t}{2} n^\s}\leq C_2 n^{\t n^{\s}}, $$
for a suitable constant $C_2>0$. Now, the relation between the sequence $M_p^{\t,\s}$ and its associated function
$ T_{\t,\s}$ given by \eqref{asocirana} implies that
\begin{equation}
\label{PWsaT}
|\widehat f(\eta)|
\leq C_2 \inf_{n\in \N_0}\frac{n^{\t n^\s}}{|\eta|^n} = C_2 e^{- T_{\t,\s}(|\eta|)},\quad \quad |\eta| \quad \text{large enough}.
\end{equation}
Then, by the left-hand side of \eqref{KonacnaocenaAsocirana} we get
$$
|\widehat f(\eta)| \leq C_3 e^{- B_\s\,\t^{-\frac{1}{\s-1}}
\left ( \ln^{\frac{\s}{\s-1}}(|\eta|) / W^{\frac{1}{\s-1}}  (\ln(|\eta|)) \right ) }
\qquad |\eta| \quad \text{large enough},
$$
with $\displaystyle C_3 =  C_2 e^{-\widetilde{B}_{\t,\s}} $. By choosing $ \t=(B_{\s}/h)^{\s-1}$, we obtain \eqref{ocenaTeorema1}, which proves the claim.
\end{proof}

In the proof of Theorem \ref{Th:first-estimate} \emph{ii)} we will use the following Lemma,
see \cite[Lemma 1, p.198]{GelfandShilov}.

\begin{lema}
\label{LemaNejednakost}
Let $\varphi \in C^\infty ({\mathbb R})$ be such that
$$
\left|x^k \varphi^{(q)}(x)\right| \leqslant C A^k B^q m_{k ,q} ,\quad k,q \in \mathbb N _0,
$$
for some $A,B,C>0$, where the numbers $m_{k,q}$ are such that
\be
\label{UslovnaNiz}
k q \frac{m_{k-1, q-1}}{m_{k ,q}} \leqslant \gamma(k+q)^\theta, \qquad  \theta \leqslant 1,
\ee
for some $\gamma>0$. Then the function $\varphi $ also satisfies
$$
\left|\left[x^k \varphi(x)\right]^{(q)}\right| \leqslant C_1 A_1^k B_1^q m_{k, q},\quad k,q\in \mathbb N _0,
$$
for some other constants $A_1,B_1,C_1>0$.
\end{lema}

\begin{rem}
\label{RemarkNejednakost}
Let $m^{s,\t,\s}_{k,q}=k!^s q^{\t q ^{\s}}$, $0<s\leq 1$, $\t>0$, $\s>1$. Then \eqref{NizOcena} implies
$$
k q \frac{m^{s,\t,\s}_{k-1, q-1}}{m^{s,\t,\s}_{k ,q}} =k^{1-s} q \frac{M^{\t,\s}_{q-1}}{M^{\t,\s}_{q}}\leq k^{1-s}\frac{q}{(2q)^{\tau (q-1)^{\s-1}}}\leq C (k+q)^{1-s},\quad k,q\in \mathbb N _0,
$$
for some $C>1$, and hence $m^{s,\t,\s}_{k,q}$ satisfies \eqref{UslovnaNiz}.
\end{rem}

Now we are ready to prove the decay estimates of orthonormal wavelets and their derivatives. Although  \eqref{NejednakostGlavna} follows from \eqref{Eq:more-precise} we decided to present both proofs since
we use different arguments there. In fact, the estimates in \eqref{NejednakostGlavna} are in the spirit of
\cite{Hernandez} and \cite{Moritoh}, whereas  \eqref{Eq:more-precise} gives an explicit dependence
of the involved  constants on the order of derivatives of the wavelet $\psi$, underlying the strong ultra-analyticity
of $\psi $ which is not considered in  \cite{Hernandez, Moritoh}.

\begin{te} \label{Th:first-estimate}
Let $\s>1$, and let $\psi \in C^\infty ({\mathbb R})$  be an  orthonormal wavelet as in Lemma \ref{Prop:F-of-psi},
so that $\widehat{\psi} \in  \mathcal{D_{ \sigma}}({\mathbb R})$  is given by  \eqref{psi}. Then $\psi$ together with its derivatives satisfies the following estimates:
\begin{itemize}
\item[\emph{i)}] for every $n\in {\mathbb N _0}$ and $h>0$ there exists $C>0$ such that
\be
\label{NejednakostGlavna}
|\psi^{(n)}(x)|  \leq C  e^{-h \left (
 \ln^{\s}(|x|)  W^{-1} (\ln(|x|)) \right )^{1/(\s - 1)} }, \, \quad |x| \; \text{large enough},
\ee
where $W $ is the Lambert function.
\item[\emph{ii)}]
for every $h>0$ there exists $C>0$ such that
\be \label{Eq:more-precise}
|\psi^{(n)}(x)|  \leq C^{n+1} \,n!^s  e^{-h\left (
 \ln^{\s}(|x|)  W^{-1} (\ln(|x|)) \right )^{1/(\s - 1)})},
\ee
for every $0<s\leq 1$,  $n\in {\mathbb N _0}$, and large enough $ |x| $. Here $W $ denotes the Lambert function.
\end{itemize}
\end{te}

\begin{proof}
\emph{i)} Let $h>0$ be given and let $n=0$. Since
$\widehat{\psi} (\xi) = e^{i \frac{\xi}{2}} b_a (\xi)$, $\xi \in {\mathbb R}$,
where $b_a \in  \, \mathcal{D_{ \sigma}}({\mathbb R}) $ is even extension of \eqref{ba}, we can apply
\eqref{PWsaT} to obtain
\begin{align*}
|\psi (x)| = C \left|\widehat{b}_a \left(x+ \frac{1}{2}\right)\right| \leq C_{h } e^{-\frac{h}{A} T_{ \sigma} (\left| x+\frac{1}{2} \right|)}\leq C_{h } e^{-hT_{ \sigma} (\left| x \right|)}, \quad x\in \mathbb{R}.
\end{align*}
where we used \eqref{AlmostIncreasing}  for the last inequality,
and replaced $h$ by $\displaystyle \frac{h}{A}$ for suitable $A>0$. Now \eqref{NejednakostGlavna} with $n=0$ follows from
the left-hand side of \eqref{KonacnaocenaAsocirana}, cf. Proposition \ref{NASPW}.

Next we estimate the derivatives of $\psi$.
Since
$$| \psi ^{(n)} (x)| = C |(\xi^n e^{i\frac{\xi}{2}} b_a (\xi))\hat{} \, (x)| =
C |(\xi^n g (\xi))\hat{} \, (x)|,
$$
and $ g(\xi) := e^{i \frac{\xi}{2}} b_a (\xi) \in \mathcal{D}_{\sigma} ({\mathbb R})$, see the proof of Lemma \ref{Prop:F-of-psi},
it is sufficient to prove that $\xi^n g (\xi) \in \mathcal{D}_{\sigma}({\mathbb R}) $ for every $n\in\mathbb{N}_0$.

The function $g (\xi)$ is compactly supported, so the same holds for  $\xi^n g (\xi)$. Let $K$ denote the support of
$\xi^n g (\xi)$. Then  we put  $\max_K |\xi|^n = a_n$, $n\in\mathbb{N}_0$, and by the Leibnitz formula we get
\begin{align*}
|(\xi^n g (\xi))^{(p)}| &= \left| \sum_{k=0}^p \binom{p}{k} \frac{d^k}{dx^k} (\xi^n) g^{(p-k)} (\xi)  \right|\\
&\leq \sum_{k=0}^p \binom{p}{k} n (n-1) \cdots (n-(k-1)) |\xi|^{n-k} \, | g^{(p-k)} (\xi) |\\
&\leq \sum_{k=0}^p \binom{p}{k} n (n-1) \cdots (n-(k-1)) a_{n-k} \, C^{(p-k)^{\sigma}} (p-k)^{\tau (p-k)^{\sigma}}\\
&\leq A p^{\tau p^{\sigma}} C^{p^{\sigma}} \sum_{k=0}^p \binom{p}{k} n^k
 = C_1 ^{p^{\sigma}} p^{\tau p^{\sigma}},
\end{align*}
where $A = \max\{ a_n, \dots, a_{n-p} \}$, and $C_1 = 2nAC$.

Thus  $\xi^n e^{i\frac{\xi}{2}} b_a (\xi) \in\mathcal{D}_{\sigma}({\mathbb R})$
for every   $n\in {\mathbb N _0}$, and the estimate  \eqref{NejednakostGlavna}
follows from Proposition \ref{NASPW}.

\emph{ii)}
Take arbitrary $\t>0$ and fix $0<s\leq 1$. Since ${\widehat \psi}\in \D_{\frac{\t}{4},\s}({\mathbb R})$, for any $m,n\in {\mathbb N _0}$ we have
\be \label{eq:procena01}
\frac{|\xi^{n}\widehat{\psi}^{(m)}(\xi)|}{ n!^s \, m^{\frac{\t}{2}  m^{\s}}}\leq B_{\t}\sup_{m,n\in{\mathbb N _0}} \frac{A^{n} }{n!^s}\frac{B_\t^{m^{\s}} }{m^{\frac{\t}{4} m^{\s}}}\leq C_{\t}, \quad \xi \in \supp \widehat {\psi},
\ee
since $\dss \sup_{m\in \mathbb{N}_0} \frac{B_\t^{m^{\s}} }{m^{\frac{\t}{4} m^{\s}}}$ is bounded by a constant which depends on $\t$, cf. \cite[Lemma 2.3]{PTT-01}.
Thus \eqref{eq:procena01} gives
\be
|\xi^{n}\widehat{\psi}^{(m)}(\xi)|\leq C_{\t} n!^s  m^{\frac{\t}{2}  m^{\s}},\quad \xi \in K,
\ee
and from Lemma \ref{LemaNejednakost} (see also Remark \ref{RemarkNejednakost}) it follows that
$$
|(\xi^{n}\widehat{\psi}(\xi))^{(m)}|\leq C'_{\t} A_1^n B_1^m n!^s m^{\frac{\t}{2}  m^{\s}},\quad \xi\in K,
$$
for suitable $C'_{\t}, A_1, B_1>0$.Thus
\begin{multline}
\label{RacunGlavna2}
|x^{m}{\psi}^{(n)}(x)|=
{|{\mathcal F}^{-1}((\xi^{n}\widehat{\psi}(\xi))^{(m)})(x)|}\\
\leq  \int_{K}|(\xi^{n}\widehat{\psi}(\xi))^{(m)}|d\xi\leq C''_{\t}  A_1^n B_1^m n!^s m^{\frac{\t}{2}  m^{\s}}, \quad m,n\in {\mathbb N _0},\quad x\in {\mathbb R}.
\end{multline}
Next we notice that
\be
B_1^m  m^{\frac{\t}{2}  m^{\s}}=\frac{B_1^m}{m^{\frac{\t}{2}  m^{\s}}} m^{\t  m^{\s}}\leq C\,  m^{\t  m^{\s}}, \quad m,n\in {\mathbb N _0}, \nonumber
\ee
with $\dss C=\sup_{m_1\in\N_0}\frac{B_1^{m_1}}{m_1^{\frac{\t}{2}  m_1^{\s}}}=e^{T_{\t/2,\s}(B_1)}$,
that is
\be \label{eq:procena02}
B_1^m  m^{\frac{\t}{2}  m^{\s}}\leq e^{T_{\frac{\t}{2},\s} (B_1) }
\,   m^{\t  m^{\s}}, \quad m,n\in {\mathbb N _0}.
\ee

Now, from  \eqref{RacunGlavna2}, \eqref{eq:procena02}, \eqref{asocirana},
and  \eqref{AsociranaSigma} it follows that there exists $C_{\t} ''' > 0$ such that
\begin{multline}
\label{eq:procena03}
|{\psi}^{(n)}(x)|\leq (C''' _{\t})^{ (n+1)} n!^s\inf_{m\in {\mathbb N _0}}\frac{m^{\t m^{\s}}}{|x|^m} \\
= (C''' _{\t} )^{(n+1)}n!^s e^{-T_{\t,\s}(|x|)} \leq C^{n+1}\, n!^s e^{-\t^{-\frac{1}{\s-1}} T_{\s}(|x|)},
\end{multline}
for any $n\in {\mathbb N _0}$, $|x|$ large enough, and for a suitable $C>0$.

Now for any given $h>0$ we can choose $\t$ so that  $\t  \asymp h^{{1-\s}}$, and
then \eqref{Eq:more-precise} follows from \eqref{eq:procena03} and the left-hand side of \eqref{TeoremaAsocirana},
which completes the proof.
\end{proof}

We finally remark that the Dziuba\'nski-Hern\'andez construction of ONW
is recently reconsidered in the framework of Gelfand-Shilov spaces in  \cite{Rakic}.
Then the Gevrey type regularity of ONW is observed both in the Fourier transform (frequency) domain and in the
configuration  (space) domain.

\section*{Acknowledgements}
This research was  supported by the Science Fund of the Republic of
Serbia, $\#$GRANT No 2727, \emph{Global and local analysis of operators and
distributions} - GOALS. N. Teofanov and S. Tuti\'c gratefully acknowledge the financial support of the Ministry of Science, Technological Development and Innovation of the Republic of Serbia (Grants No. 451-03-66/2024-03/ 200125 $\& $ 451-03-65/2024-03/200125).

%
%

\end{document}